\documentclass[11pt,a4paper]{amsart}

\usepackage{amssymb,amscd,amsmath,young, stmaryrd,tikz-cd,hyperref}
\usepackage{mathrsfs, mathdots}
\usepackage[mathcal]{eucal}
\usepackage{float,mathabx}
\usepackage{soul,caption}
\usepackage{cancel}
\usepackage[normalem]{ulem}
\usepackage{hyperref}
\usetikzlibrary{arrows}

\usepackage{longtable,array,multirow,makecell,graphicx}

\theoremstyle{plain}
\newtheorem{thm}{Theorem}[section]
\newtheorem{lem}[thm]{Lemma}

\newtheorem{cor}[thm]{Corollary}
\newtheorem*{claim*}{Claim}

\theoremstyle{remark}

\newtheorem{rem}[thm]{Remark}

\newtheorem{dfn}[thm]{Definition}
\newtheorem*{acknowledgements}{Acknowledgements}

\numberwithin{equation}{section}

\numberwithin{table}{section}

\newcommand{\N}{\mathbb{N}}
\newcommand{\Z}{\mathbb{Z}}
\newcommand{\Q}{\mathbb{Q}}

\newcommand{\tensor}{\otimes}

\renewcommand{\epsilon}{\varepsilon}

\renewcommand{\phi}{\varphi}
\renewcommand{\theta}{\vartheta}

\newcommand{\ideal}{\triangleleft}

\newcommand{\Zp}{\mathbb{Z}_{p}}

\def \mcL {\ensuremath{\mathcal{L}}}

\def \Fp {\ensuremath{\mathbb{F}_p}}

\def \Zp  {\mathbb{Z}_p}

\author{Seungjai Lee} \address{Fakult\"at f\"ur Mathematik,
	Universit\"at Bielefeld, D-33501 Bielefeld, Germany}
\email{seungjai.lee@math.uni-bielefeld.de}

\begin{document}
	
	\title[$\zeta_{G}^{\ideal}(s)$ for small $\mathfrak{T}_2$-groups and their behaviour on Residue Classes]{Normal Zeta functions of small $\mathfrak{T}_2$-groups and their behaviour on Residue Classes}

	\date{\today}
	\begin{abstract} Let $G$ be a finitely generate nilpotent class-2 torsion-free group. We study how the zeta function enumerating normal subgroups of G varies on residue classes. In particular, we show that for small such $G$ of Hirsch length less than or equal to 7, the normal zeta functions are generically always rational functions on residue classes. We then show that there are examples of groups with Hirsch length 8 whose normal zeta function is not a rational function on residue classes. We observe the connection to Higman's PORC conjecture.
	\end{abstract}
	
	\maketitle
	\setcounter{tocdepth}{1} \tableofcontents{}
	
	\thispagestyle{empty}

\section{Introduction}
\subsection{Backgrounds and motivations}
Let $G$ be a finitely generated group. In their seminal paper \cite{GSS/88}, Grunewald, Segal, and Smith defined a \textit{zeta function of a group $G$} by associating a Dirichlet series
\[\zeta_{G}^{*}(s)=\sum_{H*G}|G:H|^{-s}=\sum_{m=1}^{\infty}a_{m}^{*}(G)m^{-s},\]
where $*\in\{\leq,\triangleleft\}$ and \[a_{m}^{*}(G)=|\{H*G\mid|G:H|=m\}|.\] We call $\zeta_{G}^{\leq}(s)$ the \textit{subgroup zeta function of $G$} and $\zeta_{G}^{\triangleleft}(s)$ the \textit{normal subgroup zeta function of $G$}. 

When $G$ is a torsion-free finitely generated nilpotent group (a $\mathfrak{T}$-group for short), the nilpotency of $G$ lends itself to a natural Euler product $\zeta_{G}^{*}(s)=\prod_{p\,\textrm{prime}}\zeta_{G,p}^{*}(s)$,
where 
\[\zeta_{G,p}^{*}(s)=\sum_{i=0}^{\infty}a_{p^{i}}^{*}(G)p^{-is}.\]
We call $\zeta_{G,p}^{\leq}(s)$ the \textit{local subgroup zeta function of $G$} and $\zeta_{G,p}^{\triangleleft}(s)$ the \textit{local normal subgroup zeta function of $G$}. \cite[Theorem 3.5]{GSS/88} proved that the local zeta function $\zeta_{G,p}^{*}(s)$ is a rational function in $p^{-s}$. 

\subsubsection{Local zeta functions and the $\Fp$ points of algebraic varieties}

 One of the main questions in the study of zeta functions of groups is how $\zeta_{G,p}^{*}(s)$ vary with the prime $p$. du Sautoy and Grunewald proved \cite[Theorem 1.6]{duSG/00} that for $*\in\{\leq,\,\vartriangleleft\},$ there exist finitely many varieties $U_{1}^{*},\ldots,U_{h}^{*}$ defined over $\mathbb{Q},$  and rational functions $W_{1}^{*}(X,Y),\ldots,W_{h}^{*}(X,Y)\in\mathbb{Q}(X,Y)$ such that, for almost all primes $p$, 
 \begin{equation*}
 \zeta_{G,p}^{*}(s)=\sum_{i=1}^{h}|\overline{U_{i}^{*}}(\Fp)|W_{i}^{*}(p,t).
 \end{equation*} 

This theorem connects the structure of $\zeta_{G,p}^{*}(s)$ over different primes to the problem of counting $\Fp$ points of certain algebraic varieties, which is in general very wild and not polynomials in $p$. For example, let $E$ denote the elliptic curve $y^2=x^3-x$. du Sautoy showed in \cite{duS-ecI/01} that there exists a class-2 $\mathfrak{T}$-group $G_{E}$ with Hirsch length $h(G_{E})=9$ such that
\[\zeta_{G_E,p}^{*}(s)=W_{1}^{*}(p,t)+|E(\Fp)|W_{2}^{*}(p,t),\]
where $W_{1}^{*}(X,Y),W_{2}^{*}(X,Y)\in\Q(X,Y)$ and $|E(\Fp)|$ denotes the number of $\Fp$ points on the elliptic curve $E$. 

Many studies have been made to describe and investigate which varieties can appear in the decomposition (\cite{duS-ennui/02,Voll/05,duSWoodward/08,Voll/11} provides good surveys of this topic). 

\subsubsection{Connection to Higman's PORC conjecture}

Let $f_{n}(p)$ denote the number of $p$-groups of order $p^{n}$. Higman's PORC conjecture predicts that $f_{n}(p)$ is a PORC (Polynomial On Residue Classes) function: for a fixed $n$, there exists a fixed integer $N$ and finitely
many polynomials $g_{i}(x)\;(i=1,2,\ldots,N)$ such that if
$p\equiv i\bmod N$ then 
\[
f_{n}(p)=g_{i}(p).
\]
This conjecture is only proven to be true for $n\leq7$, and still open for all other $n$.

In \cite{duS/02}, du Sautoy discovered that the study of behaviours of local normal subgroup zeta functions of $\mathfrak{T}$-groups when $p$ varies, the \textit{uniformity problem}, is closely related to Higman's PORC conjecture:
\begin{dfn}
	The zeta functions $\zeta_{G}^{*}(s)$ are said to be \textit{finitely uniform} if there exist rational functions $W_{1}^{*}(X,Y),\ldots,W_{k}^{*}(X,Y)\in\Q(X,Y)$ for $k\in\mathbb{N}$ such that, for every prime $p$, 
	\[\zeta_{G,p}^{*}(s)=W_{i}^{*}(p,t)\]
	for some $i\in[k]$. It is said to be \textit{uniform} if $k=1$, and \textit{non-uniform} if it is not finitely uniform.  
\end{dfn}
 With this observation du Sautoy and Vaughan-Lee \cite{DuSVL/2012} proved that there
exists a parametrised family of groups, one for each prime $p> 3$,
of order $p^{9}$ whose number
of immediate descendants of order $p^{10}$ is not PORC.

Motivated by these results the author constructed another parametrised family $G_{p}$ of order $p^8$ and of exponent $p$ in \cite{Lee/2016} with the following presentation for $p>3$ whose number ot immediate descendants of order $p^{9}$ is not PORC:
\begin{thm}[\cite{Lee/2016}, Theorem 1.3] \label{thm:G_p} Let
\begin{align*}
G_{p}=\left\langle \begin{aligned}x_{1},\,x_{2},\,x_{3},\,x_{4},\,x_{5},\,x_{6},\,x_{7},\,x_{8}\mid[x_{1},\,x_{4}]=[x_{2},\,x_{5}]=x_{6},\\
[x_{1},\,x_{5}]^{2}=[x_{3},\,x_{4}]=x^{2}_{7},\,[x_{2},\,x_{4}]=[x_{3},\,x_{5}]=x_{8}
\end{aligned}
\right\rangle 
\end{align*}
where all other commutators are defined to be 1 and $g^p=1$ for all $g\in G_{p}$.  The number of immediate descendants of $G_{p}$ of order $p^{9}$ is not PORC.
\end{thm}

\subsection{Main results and organisations}

The aim of this paper is to initiate the investigation on how $\zeta_{G,p}^{\vartriangleleft}(s)$
depends on residue classes:

\begin{dfn}\label{dfn:RFORC}
	Let $G$ be a $\mathfrak{T}$-group. We say that the zeta function $\zeta_{G}^{*}(s)$ is \textit{rational functions on residue classes} (RFORC) if there exists a fixed integer $N$ and finitely many rational functions $W_{i}(X,Y)\in\Q(X,Y)$ for $i=1,\ldots,N$ such that if
	$p\equiv i\mod N$ then 
	\[\zeta_{G,p}^{*}(s)=W_{i}(p,p^{-s}).\] 
	We say $\zeta_{G}^{*}(s)$ is non-RFORC if these is no such $N$.
\end{dfn}

In Section 3, we prove Theorem \ref{thm:RFORC7}, that for any class-2 $\mathfrak{T}$-group $G$ (a $\mathfrak{T}_2$-group for short) whose Hirsch length $h(G)\leq7$ and the Pfaffian hypersurface associated to $G$ is smooth and contains no lines (cf. Section 3), its normal subgroup zeta function $\zeta_{G}^{\vartriangleleft}(s)$ is RFORC.  Then, in Section 4, by explicit calculations we prove Theorem \ref{thm:gnp8} and Theorem \ref{thm:gnp8'} that there exist   $\mathfrak{T}_2$-groups $G$ of Hirsch length $h(G)=8$ such that  $\zeta_{G}^{\vartriangleleft}(s)$  are non-RFORC. So $h(G)\leq7$ is a strict bound. Section 2 will provide some preliminary results from number theory that allow us to prove these theorems.

One should realised that this is very similar but slightly different to the Uniformity problem, since you can have finitely uniform but still non-RFORC zeta functions (e.g. Theorem \ref{thm:gnp8} and Theorem \ref{thm:gnp8'} ). Since Higman's PORC conjecture refers specifically to the dependence of $f_n(p)$ on residue classes, studying how  $\zeta_{G}^{\vartriangleleft}(s)$ varies on residue classes might provide better insight on Higman's PORC conjecture then their uniformity.

\begin{rem}
	Note that we could not call $\zeta_{G}^{\triangleleft}(s)$ ``PORC" since they are not polynomials but rational functions. However, PORC and RFORC morally mean the same thing.
\end{rem} 
\subsection{Methodology}          
 The advantage of working with $\mathfrak{T}$-groups is the use of Mal'cev correspondence. Let $L$ be a Lie ring additively isomorphic to $\Z^{n}$ for some $n$. 
In  \cite{GSS/88} Grunewald, Segal and Smith analogously defined the \textit{subalgebra} and \textit{ideal  zeta functions of $L$} to be the Dirichlet generating series
\[\zeta_{L}^{*}(s):=\sum_{m=1}^{\infty}a_{m}^{*}(L)m^{-s},\]
where $s$ is a complex variable, with the Euler decomposition $\zeta_{L}^{*}(s)=\prod_{p\textrm{ prime}}\zeta_{L(\Zp)}^{*}(s)$,
where 
\[\zeta_{L(\Zp)}^{*}(s)=\zeta_{L\tensor\Zp}(s)=\sum_{i=0}^{\infty}a_{p^{i}}^{*}(L)p^{-is}\] are the \textit{local subalgebra} ($*=\leq$) and \textit{local ideal} ($*=\ideal$)  \textit{zeta functions of $L$}.

For each  $\mathfrak{T}$-group $G$, via the Mal'cev correspondence \cite[Theorem 4.1]{GSS/88} there exists a corresponding Lie $Q$-algebra $\mcL_G$ such that 
\[\zeta_{G,p}^{*}(s)=\zeta_{\mcL_G(\Zp)}^{*}(s)\]
for all but finitely many primes $p$.  This allows us to translate the computation of local  zeta functions of $\mathfrak{T}$-groups into that of local zeta functions of Lie rings. Most of results achieved here are actually came from studying the corresponding Lie rings.

\section{Preliminary results from Number Theory}
In this section we provide some preliminary results from number theory that we use later. As customary in number theory we represent a quadratic form over a ring $R$ by specifying a homogeneous polynomial of degree 2. In \cite{LN/1996} Lidl and Niederreiter proved the following theorem. 

\begin{thm}[\cite{LN/1996}, Theorem 6.21]\label{thm:quad.diag}
	Let $w$ be a quadratic form over $\Fp$, where $p$ is an odd prime. Then $w$ is equivalent to a diagonal quadratic form $a_1x_1^2+\cdots+a_my_m^2$.
\end{thm}
If the quadratic form $w\in\Fp[y_1,\ldots,y_n]$ is equivalent to a diagonal quadratic form $a_1x_1^2+\cdots+a_ny_n^2$, some of $a_{i}$'s might be 0. If none of them are 0, in other words the rank of $w$ is $n$, then we say $w$ is nondegenerate.

\begin{thm} \label{thm:quad} Let $w=w(y_{1},\,\ldots,\,y_{n})$ be a quadratic
	form in $n$ variables $y_{1},\,\ldots,\,y_{n}$ over $\mathbb{Z}.$
	Let 
	\[
	n_{w}(p)=\left|\left\{ \boldsymbol{y}\in\mathbb{P}^{n-1}(\mathbb{F}_{p})\mid\overline{w(\boldsymbol{y})}=0\right\} \right|,
	\]
	then $n_{w}(p)$ is PORC. \end{thm}

\begin{proof} Instead of $n_{w}(p),$ we start with 
	\[
	F_{w}(p)=\left|\left\{ \boldsymbol{y}\in\mathbb{F}_{p}^{n}\mid\overline{w(\boldsymbol{y})}=0\right\} \right|.
	\]
	
	The trick is that for any given $w$ there are only finitely many primes $p$ such that $p$ is even or $p\mid D=\text{det}(w)$. Therefore, to prove PORC one only needs to consider odd primes which do not divide $D$. For these primes, since they are odd, Theorem \ref{thm:quad.diag} showed that any nonzero quadratic form $w\in\Fp[y_1,\ldots,y_n]$ would be equivalent to a diagonal quadratic form  $a_1x_1^2+\cdots+a_my_m^2$ where $1\leq m \leq n$ and all $a_{i}\neq0$. Since the number of solutions of $a_1x_1^2+\cdots+a_my_m^2=0$ in $\Fp^n$ is $p^{n-m}$ times the number of solutions of $a_1x_1^2+\cdots+a_my_m^2=0$ in $\Fp^m$, it suffices to consider the case where $n=m$, that is, where $w$ is nondegenerate. In this case Theorem 6.26 and 6.27 in \cite{LN/1996} then tells us that
	\[
	F_{w}(p)=\begin{cases}
	p^{n-1} & \textrm{if \ensuremath{n} is odd},\\
	p^{n-1}+\left(\frac{(-1)^{n/2}D}{p}\right)(p-1)(p^{\frac{n}{2}-1}) & \text{if \ensuremath{n} is even}
	\end{cases}
	\]
	where $\left(\frac{a}{p}\right)$ is the Legendre symbol. Finally, we have 
	\[
	n_{w}(p)=\frac{F_{w}(p)-1}{p-1}=\begin{cases}
	p+p^{2}+\cdots+p^{n-2} & \textrm{if \ensuremath{n} is odd},\\
	p+p^{2}+\cdots+p^{n-2}+\left(\frac{(-1)^{n/2}D}{p}\right)p^{\frac{n}{2}-1} & \textrm{if \ensuremath{n} is even }
	\end{cases}
	\]
	as required.
\end{proof}

The other important result is the following: 

\begin{thm}[\cite{Cox2013}, \cite{Weinstein/2016}] 
	Let $f(x)=x^{3}-2$. The splitting behaviour of $f(x)$ modulo $p$ is not determined by congruence conditions on $p$. Furthermore, $f(x)$ splits modulo $p$ if and only if $p\equiv 1\mod3$ and $p=a^{2}+27b^{2}$ for integers $a$ and $b$.
\end{thm}
\begin{proof}
	This follows from Theorem 4.15, Theorem 9.8 in \cite{Cox2013} and Theorem 2.2.1, Theorem 2.2.3 in \cite{Weinstein/2016}.  
\end{proof}

\begin{cor} \label{cor:noofroot}
	Let $n(p)$ denote the number of roots of $x^3-2$ modulo $p$. Then
	
	\[n(p)=\left\{
	\begin{array}{ll}
	0 & \text{if $p=1\bmod3$ and $p\neq a^{2}+27b^{2}$ for integers $a$ and $b$,} \\
	3 & \text{if $p=1\bmod3$ and $p=a^{2}+27b^{2}$ for integers $a$ and $b$,}\\
	1 & \text{if $p=2\bmod3$.}
	\end{array}
	\right.\]
	Furthermore, $n(p)$ is not PORC.
\end{cor}

\section{Normal zeta functions of $\mathfrak{T}_{2}$-groups with Hirsch length $H(G)\leq7$}

Throughout the rest of this paper, whenever there is a presentation for a group or a ring, we always assume that all other unlisted commutators are trivial.

Let $G$ be a $\mathfrak{T}_{2}$-group with centre $Z(G)$ and derived group $G':=[G,G]$. For simplicity we assume that $G/G'$ and $G'$ are torsion-free abelian of rank $d$ and $d'$ respectively, so $G/G'\cong \Z^{d}$ and $G'\cong\Z^{d'}$.  Indeed $G/G'$ and $G'$ are always finitely generated abelian groups, and since we prove results about all but finitely many primes, we can restrict ourselves to primes $p$ not dividing the orders of the respective torsion parts.

In \cite{Voll/04} and \cite{Voll/05}, Voll introduced a method for computing normal subgroup zeta functions of $\mathfrak{T}_{2}$-groups based
on an enumeration of vertices in the affine Bruhat-Tits building associated
to $\text{SL}_{n}(\mathbb{Q}_{p}).$ 

\begin{dfn}\label{def:pfaffian} Let $G$ be a $\mathfrak{T}_{2}$-group with $Z(G)=G'$. Suppose $G$ has a presentation 
	\begin{equation}
	G=\left\langle x_{1},\,\ldots,\,x_{d},\,y_{1},\,\ldots,\,y_{d'}\mid[x_{i},\,x_{j}]=M(\boldsymbol{y})_{ij}\right\rangle ,
	\end{equation}
	where $M(\boldsymbol{y})$ is an anti-symmetric $d\times d$ matrix
	of $\mathbb{Z}-$linear forms in $\boldsymbol{y}=(y_{1},\,\ldots,\,y_{d'})$.
	 If the polynomial Pf$\left(M(\boldsymbol{y})\right):=\sqrt{\text{det}\left(M(\boldsymbol{y})\right)}$
	is not identically zero we call the hypersurface 
	\[
	\mathcal{P}_{G}:=(\textrm{Pf}(M(\boldsymbol{y}))=0)
	\]
	in $\mathbb{P}^{d'-1}$ the \textit{Pfaffian hypersurface} associated
	to $G.$ Conversely, every such matrix $M(\boldsymbol{y})$ can define
	a $\mathfrak{T}_{2}$-group $G$ via (2.1). \end{dfn}

\begin{thm}[\cite{Voll/05}, Theorem 3]\label{thm:Voll.class2} Assume that $\text{Pf}\left(M(\boldsymbol{y})\right)\in\mathbb{Z}[\boldsymbol{y}]$
	is non-zero and irreducible. Assume that the Pfaffian hypersurface
	$\mathcal{P}_{G}$ is smooth and contains no lines. For a prime $p$
	let 
	\[
	n_{\mathcal{P}_{G}}(p)=\left|\mathcal{P}_{G}(\mathbb{F}_{p})\right|
	\]
	denote the number of $\mathbb{F}_{p}$-rational points of $\mathcal{P}_{G}.$
	Then there are (explicitly determined) rational functions $W_{0}(X,\,Y),\,W_{1}(X,\,Y)\in\mathbb{Q}(X,\,Y)$
	such that if $\mathcal{P}_{G}$ has good reduction $\text{mod}\,p,$
	\begin{equation}
	\zeta_{G,\,p}^{\vartriangleleft}(s)=W_{1}(p,\,p^{-s})+n_{\mathcal{P}_{G}}(p)W_{2}(p,\,p^{-s}).
	\end{equation}

\end{thm} 

 In this Section, by adapting this method, we prove Theorem \ref{thm:RFORC7}

\begin{thm}\label{thm:RFORC7} Let $G$ be a $\mathfrak{T}_{2}$-group
	where the Pfaffian hypersurface associated to $G$ is smooth and contains no lines. For $G$ of $h(G)\leq7,$ its normal
	zeta functions are always RFORC. \end{thm}

To achieve this we start with a simple but crucial Lemma:

\begin{lem}\label{lem:centre.ok}
	Let $G$ be a $\mathfrak{T}_{2}$-group where $Z(G)=G'$ and satisfies the conditions in Theorem \ref{thm:Voll.class2}. If $\zeta_{G}^{\ideal}(s)$ is RFORC, then so is $\zeta_{G\times\Z^r}^{\ideal}(s)$ for any $r\in\N$. 
\end{lem}
\begin{proof}
By Theorem \ref{thm:Voll.class2}, 
\begin{equation}\label{Gporc}
\zeta_{G,\,p}^{\vartriangleleft}(s)=W_{1}(p,\,p^{-s})+\underline{n_{\mathcal{P}_{G}}(p)}W_{2}(p,\,p^{-s}).
\end{equation}
Hence $\zeta_{G}^{\ideal}(s)$ is RFORC if and only if $n_{\mathcal{P}_{G}}(p)$ is PORC. Now, for such $G$ we can compute $\zeta_{G,\,p}^{\vartriangleleft}(s)$ by looking at its corresponding Lie ring
$L_p:=(G/G'\oplus G')\otimes_{\Z}\Z_p$
and compute $\zeta_{L_p}^{\vartriangleleft}(s)$. For each lattice $\Lambda'\leq L_p'$ put $X(\Lambda')/\Lambda'=Z(L_p/\Lambda')$. Then (essentially \cite[Lemma 6.1]{GSS/88} and \cite{Voll/04}) we have 
\begin{align*}
\zeta_{G,\,p}^{\vartriangleleft}(s)=\zeta_{L_p}^{\vartriangleleft}(s)&=\zeta_{\Zp^d}(s)\sum_{\Lambda'\leq L_p'}|L_p':\Lambda'|^{d-s}|L_p:X(\Lambda')|^{-s}\\
&=\zeta_{\Zp}(s)\zeta_p((d+d')s-dd')\sum_{\substack{\Lambda'\leq L_p'\\\Lambda'\textrm{ maximal}}}|L_p':\Lambda'|^{d-s}|L_p:X(\Lambda')|^{-s},
\end{align*}
where a lattice $\Lambda\leq\Zp^{n}$ is maximal in its homothety class if $p^{-1}\Lambda\leq\Zp^{n}$. Let $A(p,p^{-s}):=\sum_{\substack{\Lambda'\leq L_p'\\\Lambda'\textrm{ maximal}}}|L_p':\Lambda'|^{d-s}|L_p:X(\Lambda')|^{-s}$. The key observation is that the $n_{\mathcal{P}_{G}}(p)$ only comes from $|L_p:X(\Lambda')|$, and does not depend on $|L_p':\Lambda'|$ or $d$. Therefore we get 
\begin{align}
\zeta_{G\times\Z^r,\,p}^{\vartriangleleft}(s)&=\zeta_{\Zp^{d+r}}(s)\sum_{\Lambda'\leq L_p'}|L_p':\Lambda'|^{d+r-s}|L_p:X(\Lambda')|^{-s}\notag\\
&=\zeta_{\Zp^{d+r}}(s)\zeta_p((d+d')s-(d+r)d')\sum_{\substack{\Lambda'\leq L_p'\\\Lambda'\textrm{ maximal}}}|L_p':\Lambda'|^{d+r-s}|L_p:X(\Lambda')|^{-s},\notag\\
&=W_{1}'(p,\,p^{-s})+\underline{n_{\mathcal{P}_{G}}(p)}W_{2}'(p,\,p^{-s})\label{Gcentreporc}.
\end{align}
Since $n_{\mathcal{P}_{G}}(p)$ in \eqref{Gporc} and \eqref{Gcentreporc} are the same, $\zeta_{G\times\Z^r}^{\ideal}(s)$ is also RFORC for any $r\in\N$. 
\end{proof}

Lemma \ref{lem:centre.ok} implies that we only need to prove Theorem \ref{thm:RFORC7} for the case $Z(G)=G'$. The strategy for proving Theorem \ref{thm:RFORC7} is to look at the behaviour of $\zeta_{G,\,p}^{\vartriangleleft}(s)$
for different values of $d$ and $d'$. Let us call $G$ a $(d,d')$-group
if $h(G/G')=d$ and $h(G')=d'.$ As we are only looking at the $\mathfrak{T}_{2}$-groups
$G$ of $h(G)\leq7,$ the possible cases are (2,1), (3,1), (4,1),
(5,1), (6,1), (3,2), (4,2), (5,2), (3,3), (4,3).

\subsection{$d'=1$}

When $d'=1$, (2,1), (4,1) and (6,1) are the central products of 1, 2 and 3 copies of the Heisenberg group $H$ respectively, and their local normal zeta functions are given in \cite[Theorem 2.22]{duSWoodward/08} and \cite{GSS/88}. They are RFORC.

\subsection{$d'=2$}

For $d'=2$ recall Voll's method in \cite{Voll/04}. 

\begin{thm}[\cite{Voll/04}, Theorem 2] Let $G$ be a $\mathfrak{T}_{2}$-group
	with derived group $G'$ of Hirsch length 2. Then there
	are irreducible polynomials $f_{1}(t),\,\ldots,\,f_{m}(t)\in\mathbb{Q}[t]$
	and rational functions $W_{I}(X,\,Y),\,I\subseteq\{1,\,\ldots,\,m\}$
	such that for almost all primes $p$ 
	\[
	\zeta_{G,\,p}^{\vartriangleleft}(s)=\sum_{I\subseteq\{1,\,\ldots,\,m\}}c_{p,\,I}W_{I}(p,\,p^{-s}),
	\]
	where 
	\[
	c_{p,\,I}=\left|\left\{ x\in\mathbb{P}^{1}(\mathbb{F}_{p}):f_{i}(x)\equiv0\;\text{mod}\;p\textrm{\ if and only if }i\in I\right\} \right|.
	\]
	In particular, $\zeta_{G}^{\vartriangleleft}(s)$ is finitely
	uniform. \end{thm} 

Although it proves finite uniformity for $d'=2,$ it does not prove whether
$\zeta_{G}^{\vartriangleleft}(s)$ is RFORC or not. The key observation
here is that the polynomials $f_{i}$ come from either the anti-symmetric
$d\times d$ matrix $M(\boldsymbol{y})$ or its decomposed parts,
and as long as $d\leq5$ (which is our case) the behaviour of $c_{p,\,I}$
is PORC.

To make this paper self-contained we recall the definition of indecomposable groups and Theorem
6.2 and 6.3 in \cite{GSegal/84}. 

\begin{dfn} Let $G$ be a radicable $\mathfrak{T}_{2}$-group of finite Hirsch length with centre of Hirsch length 2, a so-called
	$\mathfrak{D}^{*}$-groups. A \textit{central decomposition} of $G$
	is a family $\{H_{1},\,\ldots,\,H_{m}\}$ of subgroups of $G$ such
	that:
	\begin{enumerate}
		\item $Z(H_{i})=Z(G)$ for each $i;$ 
		\item $G/Z(G)$ is the direct product of the subgroups $H_{i}/Z(G);$ and 
		\item $[H_{1},\,H_{j}]=1$ whenever $i\neq j.$ 
	\end{enumerate}
	The group $G$ is \textit{(centrally) indecomposable} if the only
	such decomposition is $\{G\}.$\end{dfn} 
\begin{thm}[\cite{GSegal/84}, Theorem 6.2] Every $\mathfrak{D}^{*}$-group
	$G$ has a central decomposition into indecomposable constituents,
	and the decomposition is unique up to an automorphism of $G.$ In
	particular, the constituents are unique up to isomorphism. 
\end{thm}

\begin{thm}[\cite{GSegal/84}, Theorem 6.3]\label{thm:D*} (i) Let $G$ be an indecomposable $\mathfrak{D}^{*}$-group
	of Hirsch length $n+2.$ Then, with respect to a suitable basis of
	$G/Z(G)$ and a suitable basis $(y_{1},\,y_{2})$ of $Z(G),$ the
	alternating bilinear map 
	\[
	\begin{aligned}\phi_{G}:G/Z(G)\times G/Z(G) & \rightarrow Z(G)\\
	(aZ(G),\,b(Z(G)) & \mapsto[a,\,b]
	\end{aligned}
	\]
	is represented by a matrix $M(\boldsymbol{y})$ as follows:
\end{thm} 
\begin{itemize}
	\item \textit{$n=2r+1.$ 
		\[
		M(\boldsymbol{y})=M_{0}^{r}(\boldsymbol{y})=\left(\begin{array}{cc}
		0 & B\\
		-B^{t} & 0
		\end{array}\right),
		\]
		where 
		\[
		B=B(\boldsymbol{y})=\left(\begin{array}{ccccc}
		y_{2} & 0 & 0 & \ldots & 0\\
		y_{1} & y_{2} & 0 & \ldots & 0\\
		0 & y_{1} & y_{2} & \ldots & 0\\
		0 & 0 & y_{1} & \ldots & 0\\
		0 &  &  & \vdots\\
		0 & 0 & 0 & \ldots & y_{2}\\
		0 & 0 & 0 & \ldots & y_{1}
		\end{array}\right)_{(r+1)\times r}
		\]
	} 
	\item \textit{$n=2r.$ 
		\[
		M(\boldsymbol{y})=M_{(f,\,e)}(\boldsymbol{y})=\left(\begin{array}{cc}
		0 & B\\
		-B^{t} & 0
		\end{array}\right),
		\]
		where 
		\[
		B=B(\boldsymbol{y})=\left(\begin{array}{cccccc}
		y_{1}+a_{1}y_{2} & y_{2} & 0 & 0 & \ldots & 0\\
		-a_{2}y_{2} & y_{1} & y_{2} & 0 & \ldots & 0\\
		a_{3}y_{2} & 0 & y_{1} & y_{2} & \ldots & 0\\
		-a_{4}y_{2} & 0 & 0 & y_{1} & \ldots & 0\\
		\vdots &  &  &  & \vdots\\
		(-1)^{r}a_{r-1}y_{2} & 0 & 0 & 0 & \ldots & y_{2}\\
		(-1)^{r+1}a_{r}y_{2} & 0 & 0 & 0 & \ldots & y_{1}
		\end{array}\right)_{r\times r}
		\]
		and $\text{det}(B(\boldsymbol{y}))=g(y_{1},\,y_{2})=y_{1}^{r}+a_{1}y_{1}^{r-1}y_{2}+\cdots+a_{r}y_{2}^{r}\in\mathbb{Q}[y_{1},\,y_{2}]$
		is such that $g(y_{1},\,1)$ is primary, say $g=f^{e}$ for $f$ irreducible
		over $\mathbb{Q},$ $e\in\mathbb{N}.$} 
\end{itemize}
\textit{(ii) If $G$ is any $\mathfrak{D}^{*}$-group, then with respect
	to a suitable basis as above, $\phi_{G}$ is represented by the diagonal
	sum of matrices like $M(\boldsymbol{y})$ above. }

Since we assume $G'=Z(G),$ these are precisely the cases where $d'=2.$
We start from $(d,d')=(3,2)$. For (3,2), Theorem \ref{thm:D*} implies that $G$ has to be an indecomposable $\mathfrak{D}^{*}$-group. Hence \cite[Proposition 2]{Voll/04} gives
\[
\zeta_{G,\,p}^{\vartriangleleft}(s)=\zeta_{\mathbb{Z}_{p}^{3}}(s)\zeta_{p}(5s-6)\zeta_{p}(3s-4)W(p,\,p^{-s})
\]
where $W(X,\,Y)=1+X^{3}Y^{3}.$

For $(d,d')=(4,2)$, if $G$ is an indecomposable $\mathfrak{D}^{*}$- group then
with respect to a suitable basis of $G/G'$ and a suitable basis $(y_{1},\,y_{2})$
of $G',$ we have
\[
M(\boldsymbol{y})=\left(\begin{array}{cccc}
0 & 0 & y_{1}+a_{1}y_{2} & y_{2}\\
0 & 0 & -a_{2}y_{2} & y_{1}\\
-y_{1}-a_{1}y_{2} & a_{2}y_{2} & 0 & 0\\
-y_{2} & -y_{1} & 0 & 0
\end{array}\right)
\]
where $\sqrt{\textrm{det}(M(\boldsymbol{y}))}=g(y_{1},\,y_{2})=y_{1}^{2}+a_{1}y_{1}y_{2}+a_{2}y_{2}^{2}\in\mathbb{Q}[y_{1},\,y_{2}]$
is such that $g(y_{1},\,1)$ is primary, say $g=f^{e}$ for $f$ irreducible
over $\mathbb{Q},\,e\in\mathbb{N}.$ Then if we let $p$ be a prime
unramified in $\mathbb{Q}[t]/f(t),$ \cite[Proposition 3]{Voll/04} gives
\[
\zeta_{G,\,p}^{\vartriangleleft}(s)=\zeta_{\mathbb{Z}_{p}^{4}}(s)\zeta_{p}(6s-8)\zeta_{p}(3s-5)\zeta_{p}(5s-5)\zeta_{p}(3es-(5e-1))(P_{1}(p,\,p^{-s})-n_{f}(p)P_{2}(p,\,p^{-s}))
\]
where $n_{f}(p)$ is the number of distinct linear factors in $\overline{f(t)}$
and 
\begin{align*}
P_{1}(X,\,Y) & =(1-X^{5}Y^{3})(1+X^{4}Y^{5})(1-X^{5e-1}Y^{3e})\\
P_{2}(X,\,Y) & =(1-Y)(1+Y)X^{4}Y^{3}(1-X^{5e}Y^{3e}).
\end{align*}
Now, as $f$ is only either quadratic or linear, $e$ is either 1
or 2, and $n_{p,\,f}$ is always PORC by Theorem \ref{thm:quad}. Hence $\zeta_{G,\,p}^{\vartriangleleft}(s)$
is RFORC.

If it is decomposable, then the only way to decompose such $G$ is
to decompose it into two groups $H_{1}$ and $H_{2}.$ With respect
to a suitable basis we can have 
\[
M(\boldsymbol{y})=\left(\begin{array}{cccc}
0 & y_{1}+a_{1}y_{2} & 0 & 0\\
-y_{1}-a_{1}y_{2} & 0 & 0 & 0\\
0 & 0 & 0 & y_{1}+a_{2}y_{2}\\
0 & 0 & -y_{1}-a_{2}y_{2} & 0
\end{array}\right)
\]
where 
\[
M_{H_{1}}(\boldsymbol{y})=\left(\begin{array}{cc}
0 & y_{1}+a_{1}y_{2}\\
-y_{1}-a_{1}y_{2} & 0
\end{array}\right)
\]
and 
\[
M_{H_{2}}(\boldsymbol{y})=\left(\begin{array}{cc}
0 & y_{1}+a_{1}y_{2}\\
-y_{1}-a_{1}y_{2} & 0
\end{array}\right).
\]
Thus there are irreducible polynomial $f_{1}(t)=t+a_{1},\,f_{2}(t)=t+a_{2}\in\mathbb{Q}[t]$
and rational function $W_{I}(X,\,Y),$ $I\subseteq\{1,\,2\}$ such
that for almost all primes $p$ 
\[
\zeta_{G,\,p}^{\vartriangleleft}(s)=\sum_{I\subseteq\{1,\,2\}}c_{p,\,I}W_{I}(p,\,p^{-s}),
\]
where 
\[
c_{p,\,I}=\left|\left\{ x\in\mathbb{P}^{1}(\mathbb{F}_{p})\mid f_{i}(x)\equiv0\;(\text{mod}\;p)\;\textrm{if and only if }i\in I\right\} \right|.
\]
Since both $f_{1}$ and $f_{2}$ are linear, we can see that $\zeta_{G,\,p}^{\vartriangleleft}(s)$
is RFORC.

For $(d,d')=(5,2)$, again if it is indecomposable \cite[Proposition 2]{Voll/04} gives
\[
\zeta_{G,\,p}^{\vartriangleleft}(s)=\zeta_{\mathbb{Z}_{p}^{5}}(s)\zeta_{p}(7s-10)\zeta_{p}(5s-6)W(p,\,p^{-s})
\]
where $W(X,\,Y)=1+X^{5}Y^{5}.$

If it is decomposable, then again it can be decomposed into $H_{1}$
and $H_{2}$ where with respect to a suitable bases we have 
\[
M(\boldsymbol{y})=\left(\begin{array}{ccccc}
0 & y_{1}+a_{1}y_{2} & 0 & 0 & 0\\
-y_{1}-a_{1}y_{2} & 0 & 0 & 0 & 0\\
0 & 0 & 0 & 0 & y_{2}\\
0 & 0 & 0 & 0 & y_{1}\\
0 & 0 & -y_{2} & -y_{1} & 0
\end{array}\right)
\]
where 
\[
M_{H_{1}}(\boldsymbol{y})=\left(\begin{array}{cc}
0 & y_{1}+a_{1}y_{2}\\
-y_{1}-a_{1}y_{2} & 0
\end{array}\right)
\]
and 
\[
M_{H_{2}}(\boldsymbol{y})=\left(\begin{array}{ccc}
0 & 0 & y_{2}\\
0 & 0 & y_{1}\\
-y_{2} & -y_{1} & 0
\end{array}\right).
\]
Now we get nothing from $H_{2}$ and just one irreducible linear polynomial
$f_{1}(t)=t+a_{1}\in\mathbb{Q}[t].$ Hence we have 
\[
\zeta_{G,\,p}^{\vartriangleleft}(s)=W_{1}(p,\,p^{-s})+c_{p}W_{2}(p,\,p^{-s})
\]
where $W_{1}(X,\,Y),\,W_{2}(X,\,Y)$ are rational functions and 
\[
c_{p}=\left\{ x\in\mathbb{P}^{1}(\mathbb{F}_{p})\mid f_{1}(x)\equiv x+a_{1}\equiv0\,(\text{mod}\;p)\right\} =1
\]
for all $p$, which is clearly PORC.

\subsection{$d'=3$}

For $(d,d')=(3,3)$, any (3,3)-group $G$ would have a presentation 
\[
\left\langle x_{1},\,x_{2},\,x_{3},\,y_{1},\,y_{2},\,y_{3}\mid[x_{1},\,x_{2}]=f_{1}(y),\,[x_{1},\,x_{3}]=f_{2}(y),\,[x_{2},\,x_{3}]=f_{3}(y)\right\rangle 
\]
where $f_{1},\,f_{2}$ and $f_{3}$ are linear polynomials in $\mathbb{Z}[y_{1},\,y_{2},\,y_{3}].$
For now, consider a corresponding Lie ring $L=L(G)=G/G'\oplus G'.$
Now, as $Z(L)=L',$ it implies 
\[
\left\langle y_{1},\,y_{2},\,y_{3}\right\rangle=\left\langle f_{1}(y),\,f_{2}(y),\,f_{3}(y)\right\rangle.
\]
Since $y_{1},\,y_{2}$ and $y_{3}$ are linearly independent, the $f_{i}$'s
are linearly independent too. Hence with a suitable base exchange
we can see that there exists a ring $R\cong L$ with a presentation
\[
\left\langle x'_{1},\,x'_{2},\,x'_{3},\,y_{1},\,y_{2},\,y_{3}\mid[x'_{1},\,x'_{2}]=y_{1},\,[x'_{1},\,x'_{3}]=y_{2},\,[x'_{2},\,x'_{3}]=y_{3}\right\rangle .
\]
In fact $R$ is a free class-2 nilpotent Lie ring on 3 generators,
and we know its local normal zeta function are RFORC (in fact they
are uniform, \cite{duSWoodward/08}, \cite{GSS/88}). As we have 
\[
\zeta_{R,\,p}^{\vartriangleleft}(s)=\zeta_{L,\,p}^{\vartriangleleft}(s)=\zeta_{G,\,p}^{\vartriangleleft}(s),
\]
this covers the case $(d,d')=(3,3)$.

Finally, suppose $G$ is a (4,3)-group. The Pfaffian hypersurface
$\mathcal{P}_{G}$ of a (4,3)-group $G$ would be a quadric form $w$
in $\mathbb{P}^{2}$ (i.e. conic). As long as $w$ is not identically
zero, smooth and contains no line, we will have 
\[
\zeta_{G,\,p}^{\vartriangleleft}(s)=W_{1}(p,\,p^{-s})+n_{w}(p)W_{2}(p,\,p^{-s})
\]
where $n_{w}(p)$ is PORC by Theorem \ref{thm:quad}. This concludes the proof
of Theorem \ref{thm:RFORC7}.

\begin{rem}
	Note that the condition on Pfaffian hypersurface associated to $G$ being smooth and containing no lines is in fact only required for $(d,d')=(4,3)$ case. We know by \cite{Beauville/00} and \cite{Voll/05} that this condition holds \textit{generically} for $d'\leq6$, which includes our case.
\end{rem}
\section{$\mathfrak{T}_{2}$-groups of Hirsch length 8}

In this section we present two different $\mathfrak{T}_{2}$-groups
$G_{np8}$ and $G_{np8'}$ of Hirsch length 8 whose local normal zeta
functions are not RFORC. The name $G_{np8}$ and $G_{np8'}$ come from "Non-PORC of Hirsch length 8".

\begin{thm} \label{thm:gnp8} Let $G_{np8}$ be a $\mathfrak{T_{2}}$-group
	with the associated Lie ring 
	
	\begin{align*}
	L(G_{np8})=\left\langle \begin{aligned}x_{1},\,x_{2},\,x_{3},\,x_{4},\,x_{5},\,y_{1},\,y_{2},\,y_{3}\mid[x_{1},\,x_{4}]=[x_{2},\,x_{5}]=y_{3},\\
	2[x_{1},\,x_{5}]=[x_{3},\,x_{4}]=2y_{1},\,[x_{2},\,x_{4}]=[x_{3},\,x_{5}]=y_{2}
	\end{aligned}
	\right\rangle 
	\end{align*}
	Then $\zeta_{G_{np8},\,p}^{\vartriangleleft}(s)$ is not RFORC.\end{thm} 

\begin{proof} We calculate $\zeta_{G_{np8},\,p}^{\vartriangleleft}(s)$
	by computing the local ideal zeta functions of the corresponding Lie
	ring 
	\[
	L(G_{np8})=\left\langle x_{1},\,\ldots,\,x_{5},\,y_{1},\,y_{2},\,y_{3}\mid[x_{i},\,x_{j}]=M(\boldsymbol{y})_{ij}\right\rangle 
	\]
	where 
	\[
	M(\boldsymbol{y})=\left(\begin{array}{ccccc}
	0 & 0 & 0 & y_{3} & y_{1}\\
	0 & 0 & 0 & y_{2} & y_{3}\\
	0 & 0 & 0 & 2y_{1} & y_{2}\\
	-y_{3} & -y_{2} & 2y_{1} & 0 & 0\\
	-y_{1} & -y_{3} & -y_{2} & 0 & 0
	\end{array}\right).
	\]
	Now this is exactly the case of \cite[Theorem 1]{Voll/04} where $d'=3$,
	except that our $d=5$ is odd. Hence we need a crucial modification.
	
	Let $L$ be a Lie ring additively isomorphic to $\Z^{d+d'}$. As shown earlier, we get
	\[
	\begin{aligned}\zeta_{L,\,p}^{\vartriangleleft}(s) & =\zeta_{\mathbb{Z}_{p}^{d}}(s)\sum_{\Lambda'\subseteq L'}\left|L':\Lambda'\right|^{d-s}\left|L:X(\Lambda')\right|^{-s}\\
	& =\zeta_{\mathbb{Z}_{p}^{d}}(s)\zeta_{p}((d+d')s-dd')A(p,\,p^{-s})
	\end{aligned}
	\]
	where 
	\[
	A(p,\,p^{-s})=\sum_{\Lambda'\subseteq L',\,\Lambda'\,\text{maximal}}\left|L':\Lambda'\right|^{d-s}\left|L:X(\Lambda')\right|^{-s}.
	\]
	Recall from
	\cite{Voll/04} the definition of the weight functions 
	\[
	\begin{aligned}w(\Lambda') & :=\log_{p}\left(\left|L':\Lambda'\right|\right)\\
	w'(\Lambda') & :=w(\Lambda')+\log_{p}\left(\left|L:X(\Lambda')\right|\right).
	\end{aligned}
	\]
	Put $T:=p^{-s}$ we get 
	\[
	A(p,\,T)=\sum_{\Lambda'\subseteq L',\,\text{maximal}}p^{dw(\Lambda')}T^{w'(\Lambda')}.
	\]
	Let us compute this generating function associated to the vertices
	of the building $\Delta_{3}$. First
	we put to use the decomposition of $\Delta_{3}$ into sector-families
	$\mathcal{S}_{F},\,F\in\mathcal{F}(p,\,3).$ We have 
	\[
	A(p,\,T)=\sum_{F\in\mathcal{F}(p,\,3)}A(p,\,T,\,F)
	\]
	where 
	\begin{align*}
	A(p,\,T,\,F):= & \frac{1}{\binom{3}{2}_{p}\binom{2}{1}_{p}}+\frac{1}{\binom{2}{1}_{p}}\sum_{[\mathbb{Z}_{p}^{3}]\neq[\Lambda']\in\partial S_{F}}p^{dw[\Lambda']}T^{w'[\Lambda']}\\
	& +\sum_{[\Lambda']\in S_{F}^{\circ}}p^{dw[\Lambda']}T^{w'[\Lambda']}
	\end{align*}
	(we use the notation $S_{F},\,\partial S_{F},\,S_{F}^{\circ}$ to
	denote the sector family indexed by the flag $F,$ its boundary and
	interior, respectively).
	
	Now, for a given lattice $\Lambda'$ of type $(p^{s+t},\,p^{t},\,1)$
	where $s,\,t\geq0$ there is a unique coset 
	\[
	\alpha G_{v}\in\text{SL}_{3}(\mathbb{Z}_{p})/G_{v},
	\]
	where $G_{v}:=\textrm{Stab}_{\text{SL}_{3}(\mathbb{Z}_{p})}\left(\mathbb{Z}_{p}^{3}\cdot\textrm{diag}(p^{s+t},\,p^{t},\,1)\right),$
	such that the admissibility condition becomes 
	\begin{align}
	\Lambda_{ab}M(\alpha^{1}) & \equiv0\;\mod\;p^{s+t}\label{st}\\
	\Lambda_{ab}M(\alpha^{2}) & \equiv0\;\mod\;p^{t},\label{t}
	\end{align}
	where we denote by $\alpha^{j}$ the $j$th column of the matrix $\alpha.$
	We now have to analyse the elementary divisors of the system of linear
	equations \eqref{st} and \eqref{t}. They only depend on how the flag $F$
	meets the degeneracy locus of the matrix $M(\boldsymbol{y}).$ Here
	comes the tricky part. If $d=2r,$ the degeneracy locus of $M(\boldsymbol{y})$
	is simply the Pfaffian hypersurface. This is why we have 
	\[
	\zeta_{G,\,p}^{\vartriangleleft}(s)=W_{1}(p,\,p^{-s})+n_{\mathcal{P}_{G}}(p)W_{2}(p,\,p^{-s})
	\]
	in \cite[Theorem. 1]{Voll/04}. However, when $d=2r+1$ (like our case),
	Pf$\left(M(\boldsymbol{y})\right):=\sqrt{\text{det}\left(M(\boldsymbol{y})\right)}$
	is identically zero so there is no Pfaffian hypersurface. So we have
	to manually identify and analyse the degeneracy locus of $M(\boldsymbol{y})$.
	In this case the degeneracy locus
	of 
	\[
	M(\boldsymbol{y})=\left(\begin{array}{ccccc}
	0 & 0 & 0 & y_{3} & y_{1}\\
	0 & 0 & 0 & y_{2} & y_{3}\\
	0 & 0 & 0 & 2y_{1} & y_{2}\\
	-y_{3} & -y_{2} & 2y_{1} & 0 & 0\\
	-y_{1} & -y_{3} & -y_{2} & 0 & 0
	\end{array}\right)
	\]
	is 
	\[
	V=\{\boldsymbol{y}\in\mathbb{P}^{2}(\mathbb{Q})\mid2y_{1}^{2}-y_{2}y_{3}=y_{2}^{2}-2y_{1}y_{3}=y_{3}^{2}-y_{1}y_{2}=0\},
	\]
	which is a zero-dimensional
	smooth subvariety in $\mathbb{P}^{2},$ namely 
	\[
	V=\{(1,\,k^{2},\,k)\mid\,k^{3}=2\}.
	\]
	Since we have identified the degeneracy locus, we compute $A(p,\,T,\,F)$ for 
	\begin{enumerate}
		\item $M(\overline{\alpha^{1})}\;\text{maximal, }$ $A(p,\,T,\,F)=:A_{\textrm{off/off}}(p,\,T),$ 
		\item $M(\overline{\alpha^{1})}\;\text{not maximal, }$ $A(p,\,T,\,F)=:A_{\textrm{sm.pt/off}}(p,\,T).$ 
	\end{enumerate}
	Start with Case (1). $M(\alpha^{1})$ always has a $(d-1)$-minor
	which is a $p$-adic unit, so
	\[
	w'([\Lambda'])=s+3t+(d-1)(s+t)=ds+(d+2)t.
	\]
	The boundary of a sector-family in $\triangle_{3}$ falls into three
	components: the root vertex, (maximal) lattices of type $(p^{s},\,1,\,1),$
	and lattices of type $(p^{t},\,p^{t},\,1).$ For $s,\,t\geq1$ there
	are $p^{2(s-1)},\,p^{2(t-1)}$ and $p^{s-1}p^{t-1}p^{s+t-1}=p^{2s+2t-3}$
	lattices in $S_{F}$ of type $(p^{s},\,1,\,1),$ $(p^{t},\,p^{t},\,1)$
	and $(p^{s+t},\,p^{t},\,1)$ respectively. Therefore we get 
	\begin{align*}
	A_{\textrm{off/off}}(p,\,T):= & \frac{1}{\binom{3}{2}_{p}\binom{2}{1}_{p}}+\frac{1}{\binom{2}{1}_{p}}\left(\sum_{s\geq1}p^{(d+2)s-2}T^{ds}+\sum_{t\geq1}p^{(2d+2)t-2}T^{(d+2)t}\right)\\
	& +\sum_{s,\,t\geq1}p^{((d+2)s+(2d+2)t-3)}T^{ds+(d+2)t}\\
	= & \frac{1+p^{d}T^{d}+p^{d+1}T^{d}+p^{2d}T^{d+2}+p^{2d+1}T^{d+2}+p^{3d+1}T^{2d+2}}{\binom{3}{2}_{p}\binom{2}{1}_{p}\left(1-p^{2d+2}T^{d+2}\right)\left(1-p^{d+2}T^{d}\right)}.
	\end{align*}
	For case (2), the matrix $M(\alpha^{1})$
	will always have a $(d-3)-$minor which is a $p$-adic unit. So by
	the same logic, we can choose affine local coordinates $(x,\,y,\,1)$
	around $\alpha^{1}\in\mathbb{P}^{2}(\mathbb{Z}_{p})$ such that 
	\[
	w'([\Lambda'])=ds+(d+2)t-2\textrm{min}(s,\,v_{p}(x),\,v_{p}(y))
	\]
	where $x,\,y\in p\mathbb{Z}/(p^{s}).$ Start with 
	\[
	\sum_{[\mathbb{Z}_{p}^{3}]\neq[\Lambda']\in\partial S_{F}}p^{dw[\Lambda']}T^{w'[\Lambda']}
	\]
	consisting of lattices of type $(p^{s},\,1,\,1),$ $s\geq1.$ The
	map 
	\[
	[\Lambda']\mapsto p^{dw[\Lambda']}T^{w'[\Lambda']}
	\]
	factorises over the set 
	\[
	N:=\{(a,\,b,\,c)\in\mathbb{N}_{>0}^{3}\mid a\geq b,\,a\geq c\},
	\]
	which we view as the intersection of $\mathbb{N}_{>0}^{3}$ with some
	closed polyhedral cone $C$ in $\mathbb{R}_{>0}^{3}$ as $\psi\phi$
	where 
	\begin{align}
	\phi:[\Lambda'] & \mapsto(s,\,v_{p}(x),\,v_{p}(y))\\
	\psi:(a,\,b,\,c) & \mapsto p^{da}T^{da-2\textrm{min}(a,\,b)}.
	\end{align}
	Note that 
	\begin{equation}
	\left|\phi^{-1}(a,\,b,\,c)\right|=\begin{cases}
	1 & \textrm{if}\;a=b=c,\\
	(1-p^{-1})p^{a-b} & \textrm{if}\;a>b,\,a=c,\\
	(1-p^{-1})p^{a-c} & \textrm{if}\;a>c,\,a=b,\\
	(1-p^{-1})^{2}p^{2a-b-c} & \textrm{if}\;a>b,\,a>c.
	\end{cases}
	\end{equation}
	We decompose the cone $N$ into sub-cones $N_{j}$ on which the values
	$\left|\phi^{-1}(a,\,b,\,c)\right|\psi(a,\,b,\,c)$ are easier to
	sum over. We choose the decomposition 
	\begin{align*}
	N & =N_{0}+N_{1}+N_{2}+N_{3}+N_{4}+N_{5}\\
	N_{0} & :=\{(a,\,b,\,c)\in N\mid a=b=c\geq1\},\\
	N_{1} & :=\{(a,\,b,\,c)\in N\mid a=c>b\geq1\},\\
	N_{2} & :=\{(a,\,b,\,c)\in N\mid a=b>c\geq1\},\\
	N_{3} & :=\{(a,\,b,\,c)\in N\mid a>c>b\geq1\},\\
	N_{4} & :=\{(a,\,b,\,c)\in N\mid a>b>c\geq1\},\\
	N_{5} & :=\{(a,\,b,\,c)\in N\mid a>b=c\geq1\}.
	\end{align*}
	Let $n_{j}:=\textrm{dim}(N_{j}).$ Table 4.1 records the generating
	functions $F_{j}(X,\,Y,\,Z)$ together with the integers $n_{j}$
	and Laurent monomials $m_{jX}(p,\,T),\,m_{jY}(p,\,T),\,m_{jZ}(p,\,T).$
	
	\begin{table}\label{tab:gen.func}
		 \centering \protect\protect\caption{}
		\begin{tabular}{|c|c|c|c|c|c|}
			\hline 
			$j$  & $n_{j}$  & $F_{j}(X,\,Y,\,Z)$  & $m_{jX}(p,\,T)$  & $m_{jY}(p,\,T)$  & $m_{jZ}(p,\,T)$\tabularnewline
			\hline 
			0  & 1  & $\frac{XYZ}{1-XYZ}$  & $p^{d}T^{d}$  & $T^{-2}$  & 1\tabularnewline
			\hline 
			1  & 2  & $\frac{X^{2}YZ^{2}}{(1-XYZ)(1-XZ)}$  & $p^{d+1}T^{d}$  & $p^{-1}T^{-2}$  & 1\tabularnewline
			\hline 
			2  & 2  & $\frac{X^{2}Y^{2}Z}{(1-XYZ)(1-XY)}$  & $p^{d+1}T^{d}$  & $1$  & $p^{-1}T^{-2}$\tabularnewline
			\hline 
			3  & 3  & $\frac{X^{3}YZ^{2}}{(1-XYZ)(1-XZ)(1-X)}$  & $p^{d+2}T^{d}$  & $p^{-1}T^{-2}$  & $p^{-1}$\tabularnewline
			\hline 
			4  & 3  & $\frac{X^{3}Y^{2}Z}{(1-XYZ)(1-XY)(1-X)}$  & $p^{d+2}T^{d}$  & $p^{-1}$  & $p^{-1}T^{-2}$\tabularnewline
			\hline 
			5  & 3  & $\frac{X^{2}YZ}{(1-XYZ)(1-X)}$  & $p^{d+2}T^{d}$  & $p^{-1}T^{-2}$  & $p^{-1}$\tabularnewline
			\hline 
		\end{tabular}
	\end{table}

	Thus 
	\begin{align*}
	\sum_{[\mathbb{Z}_{p}^{3}]\neq[\Lambda']\in\partial S_{F}}p^{dw[\Lambda']}T^{w'[\Lambda']} & =\sum_{j=0}^{5}\sum_{(a,\,b,\,c)\in N_{j}}\left|\phi^{-1}(a,\,b,\,c)\right|\psi(a,\,b,\,c)\\
	& =\left.\sum_{j=0}^{5}(1-p^{-1})^{n_{j}-1}F_{j}(X,\,Y,\,Z)\right|_{\substack{X=m_{jX}(p,\,T)\\
			Y=m_{jY}(p,\,T)\\
			Z=m_{jZ}(p,\,T)
		}
	}\\
	& =\frac{p^{d}T^{d-2}(1-p^{d}T^{d})}{(1-p^{d}T^{d-2})(1-p^{d+2}T^{d})}.
	\end{align*}
	The generating function counting over lattices of type $(p^{t},\,p^{t},\,1),$
	$T\geq1,$ is clearly the same as Case (1), i.e. is given by 
	\[
	\sum_{t\geq1}p^{2t-2}\cdot p^{2dt}T^{(d+2)t}=\frac{p^{2d}T^{d+2}}{1-p^{2d+2}T^{d+2}}.
	\]
	Finally, we count over $[\Lambda']\in S_{F}^{\circ},$ i.e. over lattices
	of type $(p^{s+t},\,p^{t},\,1),$ $s,\,t\geq1.$ This is easy, since
	the generating function equals $p$ times the product of the previous
	generating functions counting over the boundary. Thus we have 
	\[
	\sum_{[\Lambda']\in S_{F}^{\circ}}p^{dw[\Lambda']}T^{w'[\Lambda']}=p\cdot\frac{p^{d}T^{d-2}(1-p^{d}T^{d})}{(1-p^{d}T^{d-2})(1-p^{d+2}T^{d})}\cdot\frac{p^{2d}T^{d+2}}{1-p^{2d+2}T^{d+2}}
	\]
	Hence we get 
	\begin{align*}
	A_{\textrm{sm.pt/off}}(p,\,T)= & \frac{1}{\binom{3}{2}_{p}\binom{2}{1}_{p}}+\frac{1}{\binom{2}{1}_{p}}\left(\frac{p^{d}T^{d-2}(1-p^{d}T^{d})}{(1-p^{d}T^{d-2})(1-p^{d+2}T^{d})}+\frac{p^{2d}T^{d+2}}{1-p^{2d+2}T^{d+2}}\right)\\
	& +p\cdot\frac{p^{d}T^{d-2}(1-p^{d}T^{d})}{(1-p^{d}T^{d-2})(1-p^{d+2}T^{d})}\cdot\frac{p^{2d}T^{d+2}}{1-p^{2d+2}T^{d+2}}
	\end{align*}
	and 
	\begin{align*}
	A(p,\,T)= & \left(\binom{3}{2}_{p}-\left|V(\mathbb{F}_{p})\right|\right)\binom{2}{1}_{p}A_{\textrm{off/off}}(p,\,T)\\
	& +\left|V(\mathbb{F}_{p})\right|\binom{2}{1}_{p}A_{\textrm{sm.pt/off}}(p,\,T)\\
	= & A_{1}(p,\,T)+\left|V(\mathbb{F}_{p})\right|A_{2}(p,\,T),
	\end{align*}
	where 
	\begin{align*}
	A_{1}(p,\,T)= & \frac{1+p^{d}T^{d}+p^{d+1}T^{d}+p^{2d}T^{d+2}+p^{2d+1}T^{d+2}+p^{3d+1}T^{2d+2}}{\left(1-p^{2d+2}T^{d+2}\right)\left(1-p^{d+2}T^{d}\right)}\\
	A_{2}(p,\,T) & =\frac{(1-T)(1+T)p^{d}T^{d-2}(1+p^{2d+1}T^{d+2})}{(1-p^{d}T^{d-2})(1-p^{d+2}T^{d})(1-p^{2d+2}T^{d+2})}.
	\end{align*}
	As $d=5$ we have
	
	\begin{align*}
	\zeta_{G_{np8},\,p}^{\vartriangleleft}(s) & =\zeta_{L(G)_{p}}^{\vartriangleleft}(s)\\
	& =\zeta_{\mathbb{Z}_{p}^{5}}(s)\zeta_{p}(8s-15)A(p,\,p^{-s})\\
	& =\zeta_{\mathbb{Z}_{p}^{5}}(s)\zeta_{p}(8s-15)\zeta_{p}(7s-12)\zeta_{p}(5s-7)W_{1}(p,\,p^{-s})\\
	& +\left|V(\mathbb{F}_{p})\right|\zeta_{\mathbb{Z}_{p}^{5}}(s)\zeta_{p}(8s-15)\zeta_{p}(3s-5)\zeta_{p}(5s-7)\zeta_{p}(7s-12)W_{2}(p,\,p^{-s})
	\end{align*}
	where 
	\begin{align*}
	W_{1}(X,\,Y) & =1+X^{5}Y^{5}+X^{6}Y^{5}+X^{10}Y^{7}+X^{11}Y^{7}+X^{16}Y^{12}\\
	W_{2}(X,\,Y) & =(1-Y)(1+Y)X^{5}Y^{3}(1+X^{11}Y^{7}).
	\end{align*}

	From Corollary \ref{cor:noofroot}, 
	\[
	\left|V(\mathbb{F}_{p})\right|=\left|\{(1,\,k^{2},\,k)\}\right|=\begin{cases}
	3 & p\equiv\textrm{1 mod 3 and \ensuremath{p=a^{2}+27b^{2}},}\\
	0 & p\equiv\textrm{1 mod 3 and \ensuremath{p\neq a^{2}+27b^{2}},}\\
	1 & p\equiv\textrm{2 mod 3},
	\end{cases}
	\]
	for integers $a$ and $b$, and is not PORC. Thus  $\zeta_{G_{np8},\,p}^{\vartriangleleft}(s)$
	is Non-RFORC as required. \end{proof}
\begin{rem}
	Note that $\zeta_{G_{np8},\,p}^{\vartriangleleft}(s)$ is still finitely uniform. In fact, this group $G_{np8}$ can be seen as a $\Zp$ version of the finite $p$-group $G_p$ of order $p^8$ that has been introduced in Theorem \ref{thm:G_p} (\cite[Theorem 1.3]{Lee/2016}), which gave non-PORC number of immediate descendants of order $p^9$. This demonstrates that the PORC conjecture can be better studied by the RFORC behaviour of normal zeta functions.
\end{rem}

We can encode the same variety $V$ (in a different form)
into another $\mathfrak{T}_{2}$-group $G_{np8'}$ of Hirsch length
$8$ such that the behaviour of its normal local factor is also governed
by $\left|V(\mathbb{F}_{p})\right|.$ 
\begin{thm} \label{thm:gnp8'} Let $G_{np8'}$
	be a $\mathfrak{T}_{2}$-group whose associated Lie ring $L(G_{np8'})$
	has a presentation 
	\[
	L(G_{np8'})=\left\langle x_{1},\,x_{2},\,x_{3},\,x_{4},\,x_{5},\,x_{6},\,y_{1},\,y_{2}:[x_{i},\,x_{j}]=M(\boldsymbol{y})\right\rangle 
	\]
	where 
	\[
	M(\boldsymbol{y})=\left(\begin{array}{cccccc}
	0 & 0 & 0 & y_{1} & y_{2} & 0\\
	0 & 0 & 0 & 0 & y_{1} & y_{2}\\
	0 & 0 & 0 & 2y_{2} & 0 & y_{1}\\
	-y_{1} & 0 & -2y_{2} & 0 & 0 & 0\\
	-y_{2} & -y_{1} & 0 & 0 & 0 & 0\\
	0 & -y_{2} & -y_{1} & 0 & 0 & 0
	\end{array}\right).
	\]
	Then $\zeta_{G_{np8'},\,p}^{\vartriangleleft}(s)$ is Non-RFORC.\end{thm}
\begin{proof} $G_{np8'}$ is in fact an indecomposable
	$\mathfrak{D}^{*}$-group. The matrix $M(\boldsymbol{y})$ is in the from 
	\[
	\left(\begin{array}{cc}
	0 & B\\
	-B^{t} & 0
	\end{array}\right)
	\]
	in Theorem \ref{thm:D*} with $a_{1}=0,\,a_{2}=0,\,a_{3}=-2.$
	
	Now det$(B(\boldsymbol{y}))=g(y_{1},\,y_{2})=y_{1}^{3}+2y_{2}^{3}$
	and $g(y_{1},\,1)=y_{1}^{3}+2=f$ where $f$ is irreducible over $\mathbb{Q}$. Hence \cite[Proposition 3]{Voll/04} gives
	\[
	\zeta_{G,\,p}^{\vartriangleleft}(s)=\zeta_{\mathbb{Z}_{p}^{6}}(s)\zeta_{p}(8s-12)\zeta_{p}(5s-7)\zeta_{p}(7s-7)\zeta_{p}(5s-6)(P_{1}(p,p^{-s})+n_{f}(p)(P_{2}(p,p^{-s}))
	\]
	where 
	\begin{align*}
	P_{1}(X,\,Y) & =(1-X^{7}Y^{5})(1+X^{6}T^{7})(1-X^{6}Y^{5})\\
	P_{2}(X,\,Y) & =(1-Y)(1+Y)X^{6}Y^{5}(1-X^{7}Y^{5})
	\end{align*}
	and $n_{f}(p)$ is the number of distinct linear factors in $\overline{f(t)}=\overline{(y_{1}^{3}+2)}.$
	Again, 
	\[
	n_{f}(p)=\begin{cases}
	3 & p\equiv\textrm{1 mod 3 and \ensuremath{p=a^{2}+27b^{2},}}\\
	0 & p\equiv\textrm{1 mod 3 and \ensuremath{p\neq a^{2}+27b^{2,}}}\\
	1 & p\equiv\textrm{2 mod 3},
	\end{cases}
	\]
	for integers $a$ and $b$. Hence $\zeta_{G_{np8'},\,p}^{\vartriangleleft}(s)$
	is non-RFORC. \end{proof}
\begin{acknowledgements}
	This article comprises parts of the author's Ph.D thesis \cite{Lee/19thesis} from the University of Oxford. The author gratefully acknowledge inspiring mathematical discussions with Marcus du Sautoy, Roger Heath-Brown, Benjamin Klopsch, Dan Segal and Christopher Voll about
	the research presented in this paper.
\end{acknowledgements}
\bibliographystyle{amsplain}
\bibliography{Lee_RFORC_22OCT20_arXiv}
\end{document}